\documentclass[a4paper,10pt, oneside]{article}
\usepackage{amsmath,amssymb,amsthm,mathrsfs,graphicx}
\usepackage[affil-it]{authblk}
\usepackage{authblk}
\usepackage[inline]{enumitem}
\usepackage[font=small,labelfont=md,textfont=it]{caption}
\usepackage{floatrow,float}
\usepackage[titletoc, title]{appendix}
\usepackage[colorlinks,linkcolor=blue,citecolor=blue]{hyperref}
\usepackage{etoolbox}
\usepackage{longtable}
\usepackage{diagbox}
\usepackage{booktabs,makecell,multirow}
\usepackage[capitalise,nosort]{cleveref}
\usepackage{cases,color}
\usepackage{ulem}
\usepackage{pifont}
\crefname{equation}{}{}
\crefname{lemma}{Lemma}{Lemmas}
\crefname{theorem}{Theorem}{Theorems}
\crefname{discr}{Discretization}{Discretizations}
\numberwithin{equation}{section}
\apptocmd{\sloppy}{\hbadness 10000\relax}{}{}

\newcommand{\ssnm}[1]
{
  \left\vert\kern-0.25ex
  \left\vert\kern-0.25ex
  \left\vert
  {#1}
  \right\vert\kern-0.25ex
  \right\vert\kern-0.25ex
  \right\vert
}

\makeatletter
\def\spher@harm#1{%
  \vbox{\hbox{%
    \offinterlineskip
    \valign{&\hb@xt@2\p@{\hss$##$\hss}\vskip.2ex\cr#1\crcr}%
  }\vskip-.36ex}%
}
\def\gshone{\spher@harm{.}}
\def\gshtwo{\spher@harm{.&.}}
\def\gshthree{\spher@harm{.&.&.}}
\let\gsh\spher@harm
\makeatother

\newtheorem{lemma}{Lemma}[section]
\newtheorem{remark}{Remark}[section]
\newtheorem{theorem}{Theorem}[section]
\newtheorem{example}{Example}[section]

\makeatletter\def\@captype{table}\makeatother

\begin{document}
	
\title{
  \Large\bf Semi-discrete and fully discrete mixed finite element methods for  Maxwell viscoelastic model of wave propagation
\thanks
  {
    This work was supported in part by National Natural Science Foundation
    of China (11771312).
  }
}
\author{
  Hao Yuan \thanks{Email:  kobeyuanhao@qq.com},
  Xiaoping Xie \thanks{Corresponding author. Email: xpxie@scu.edu.cn} \\
  {School of Mathematics, Sichuan University, Chengdu 610064, China}
}

\date{}
\maketitle

\begin{abstract}
Semi-discrete and fully discrete mixed finite element methods are considered for  Maxwell-model-based problems of wave propagation in linear viscoelastic solid. This mixed finite element framework allows the use of a large class of existing mixed  conforming finite elements for elasticity   in the spatial discretization. In the fully discrete scheme, a Crank-Nicolson scheme is adopted for the approximation of the temporal derivatives of  stress and velocity variables. Error estimates of the semi-discrete and fully discrete schemes, as well as an unconditional stability result for the fully discrete scheme, are derived. Numerical experiments are provided to verify the theoretical results.
\end{abstract}

\medskip\noindent{\bf Keywords:} Maxwell viscoelastic model;   mixed finite element;   semi-discrete and fully discrete; error estimate; stability

\section{Introduction}
Let $\Omega\subset \mathbb{R}^d$($d$=2 or 3) be a bounded open domain with boundary $\partial\Omega$ and $T$ be a positive constant. Consider the following Maxwell viscoelastic model of wave propagation:
\begin{equation}\label{1.1}
	\left\{
	\begin{array}{ll}
		\rho \boldsymbol{u}_{tt}=\mathrm{\textbf{div}}\sigma+\boldsymbol
		{f}, & (\boldsymbol{x},t)\in\Omega\times[0,T], \\
		\sigma+\sigma_t=\mathbb{C}\varepsilon(\boldsymbol{u}_t), &(\boldsymbol{x},t)\in\Omega\times[0,T],\\
		\boldsymbol{u}=0, &  (\boldsymbol{x},t)\in\partial\Omega\times[0,T], \\
		\boldsymbol{u}(\boldsymbol{x},0)=\phi_0(\boldsymbol{x}),\boldsymbol{u}_t(\boldsymbol{x},0)=\phi_1(\boldsymbol{x}),\sigma(\boldsymbol{x},0)=\psi_0(\boldsymbol{x}),& \boldsymbol{x}\in\Omega.
	\end{array}
	\right.
\end{equation}
Here  $\boldsymbol{u}=(u_1,...,u_d)^{\mathrm{T}}$ is the displacement field, $\sigma=(\sigma_{ij})_{d\times d}$ the symmetric stress tensor,  $\varepsilon(\boldsymbol{u})=(\nabla\boldsymbol{u}+(\nabla\boldsymbol{u})^{\mathrm{T}})/2$ the strain tensor, and  $g_t:=\partial g/\partial t $ and $g_{tt}:=\partial^2 g/\partial t^2 $ for any function $g(\boldsymbol{x},t)$. $\rho(\boldsymbol{x})$ denotes the mass density,  and  $\mathbb{C}$ a  rank 4 symmetric tensor, with
\begin{align}
	&0<\rho_0\leq\rho\leq\rho_1<\infty  \quad a.e.\ x\in\Omega, \label{2.4}\\
	&0<M_0\tau:\tau \leq \mathbb{C}^{-1}\tau: \tau \leq M_1 \tau: \tau \quad \forall \text{ symmetric   tensor }\tau =(\tau_{ij})_{d\times d},    a.e.\ \boldsymbol{x}\in\Omega, \label{2.5}
\end{align}
where $ \rho_0$,$ \rho_1$, $M_0 $ and $M_1$ are four positive constants, and $\sigma:\tau : = \sum\limits_{i=1}^d\sum\limits_{j=1}^d \sigma_{ij}\tau_{ij}$.
Note that  $\mathbb{C}\varepsilon(\boldsymbol{u}_t)$ is of the  form 
\begin{align}
	\mathbb{C}\varepsilon(\boldsymbol{u}_t)=2\mu \varepsilon(\boldsymbol{u}_t)+\lambda \text{div} \boldsymbol{u}_t  I
\end{align}
for an isotropic elastic medium,  where $\mu, \ \lambda$ are  the Lam\'e parameters, and $I$ the identity matrix.  $\boldsymbol{f}=(f_1,...,f_d)^{\mathrm{T}}$ is the body force, and  $\phi_0(\boldsymbol{x}),\ \phi_1(\boldsymbol{x}),\ \psi_0(\boldsymbol{x})$ are initial data. 

Numerous materials simultaneously  display  elastic and viscous
kinematic   behaviours. Such a feature, called viscoelasticity, can be characterized by using
  springs, which obey the Hooke's law, and viscous dashpots, which obey the Newton's law.  Different combinations of   the springs and dashpots lead to various viscoelastic models, e.g. the three classical models of Zener , Voigt and Maxwell. 
We note that there is a unified framework to describe the general constitutive law of viscoelasticity by using convolution integrals in time with some kernels \cite{Christensen1982Theory,Fung1966International,Salencon2016Visco}; however,   the integral forms of constitutive laws, compared with  the differential forms,     bring   more difficulties to the design of algorithms  due to the numerical convolution integral.  We refer the reader to  \cite{1960Bland, 
2007Dill,1998Drozdov,  Fung1966International,  1988Boundary, 1962Gurtin, 2000Nonlinear, Marques2012Computational} for several monographs on the 
development  and applications  of   viscoelasticity theory.

The numerical simulation of wave propagation in viscoelastic materials was first discussed  by  Kosloff et al.
 in \cite{Carcione0Wave, Carcione01988}, 
 where  memory variables were introduced to avoid the convolutional integral in the constitutive relation. Later,  
  finite difference methods were developed  in \cite{2002Parallel,Sabinin2003Numerical,Tong1998Efficient}  for the model with memory variables.  There are 
  considerable research efforts on the finite element discretization in this field.  In  \cite{1995Numerical}  Janovsky et al. studied the continuous/discontinuous Galerkin finite element discretization and used a numerical quadrature formula  to approximate the Volterra time integral term. 
Ha et al. \cite{2002Nonconforming} proposed a nonconforming finite element method for a viscoelastic complex model in the space–frequency domain. B\'{e}cache el at. \cite{Eliane2005A} applied a  family of mixed finite elements  with mass lumping, together with   a leap-frog scheme in time discretization, to numerically solve  the Zener model,  and showed that their scheme is stable under certain CFL condition. In \cite{B2006Discontinuous,B2003Discontinuous,B2007Discontinuous}, Rivi\`{e}re et al. analyzed  discontinuous Galerkin methods  with a Crank-Nicolson temporal discrete scheme  for  quasistatic linear viscoelasticity and linear/nonlinear diffusion viscoelastic models.  Rognes and Winther \cite{2010MIXED} proposed  mixed finite element methods   for quasistatic Maxwell   and Voigt models using weak symmetry, and   used a second backward difference scheme in the full discretization. Lee \cite{Lee2012Mixed}  studied   mixed finite element methods with weak symmetry for 
the Zener ,Voigt and Maxwell models and  adopted the Crank-Nicolson scheme in temporal discretization. Severino and Guillermo \cite{Marques2012Computational} gave an overview of numerical methods for problems in viscoelasticity including finite elements, boundary elements, and finite volume formulations.   Kimura et al. \cite{2018The} studied the gradient flow structure of an extended Maxwell model with a relaxation parameter and proposed a structure-preserving P1/P0 finite element scheme.  Recently, Wang and Xie  \cite{2020Wang} analyzed a  hybrid stress finite element method for the Maxwell model and used a second order implicit difference in the fully discrete scheme.

In this paper, we consider semi-discrete and fully discrete mixed finite element discretizations for the Maxwell viscoelasticity model \cref{1.1}.   Some existing  mixed  conforming finite elements for elasticity  are applied  in the spatial discretization   to approximate the stress and velocity variables.    In the full  discretization,  the Crank-Nicolson scheme is adopted  to discretize  the temporal derivatives of   stress and velocity.  We  derive optimal error estimates   for both the semi-discrete and fully discrete schemes, and give an unconditional stability result for the fully discrete scheme.

The rest of this paper is arranged as follows. Section 2 introduces notations and weak formulations. Section 3 gives a general  mixed conforming finite element framework and carries out the error estimation of the semi-discrete scheme. Section 4 presents the fully discrete scheme and derives    stability and error estimates. Finally, numerical examples are provided in Section 5  to verify the theory results when using two low order 
rectangular elements in the spatial discretization. 

\section{Weak formulations}

We first introduce some notations.   For any nonnegative integer $r$, denote by $H^r(\Omega)$ and  $H^r_0(\Omega)$ the standard Sobolev spaces with norm $||\cdot||_r$ and semi-norm $|\cdot|_r$. In particular,  $H^0(\Omega)=L^2(\Omega)$ is the space of square integrable functions.  We adopt the convention that    an underline (or double underlines) corresponds to a vector-valued ( or tensor-valued) space.

For  any vector-valued ( or tensor-valued)  space $X$, defined on $\Omega$,  with norm $||\cdot||_X$,   denote
\begin{align*}
	L^p([0,T];X):=\left\{\boldsymbol{w}:[0,T]\rightarrow X;\ ||\boldsymbol{w}||_{L^p(X)}<\infty\right\},
\end{align*}
where
\begin{align*}
	||\boldsymbol{w}||_{L^p( X)}:=\left\{ \begin{array}{ll}
	(\int_{0}^{T}||\boldsymbol{w}(t)||_X^p)^{1/p} & \text{ if }1\leq p< \infty,\\
	\mathop{\mathrm{esssup}}_{0\leq t\leq T}||\boldsymbol{w}(t)||_X  & \text{ if } p=\infty,
	\end{array}
	\right.
\end{align*}	
and  $\boldsymbol{w}(t)$ abbreviates   $\boldsymbol{w}(\boldsymbol{x},t)$. 
For   integer $r\geq0$, the space $C^r([0,T];X)$ can be defined similarly.  In the forthcoming analysis, $X$ may be taken as $\underline{L}^2(\Omega), $ $\underline{H}^r(\Omega)$ and  
\begin{align*}
	&\uuline{\mathrm{H}}(\mathrm{\textbf{div}},\Omega,S):=\{\tau=(\tau_{ij})_{d\times d}\in  \uuline{L}^2(\Omega)| \ \tau_{ij}=\tau_{ji}, \ \textbf{div}\tau \in \underline{L}^2(\Omega)\}.
\end{align*}

For convenience, we use the notation $a\lesssim b$ to denote that there exists a generic positive constant $C$, independent of the spatial and temporal mesh parameters, $h$ and $ \Delta t$, 
such that $a\leq Cb.$

We also need  two Gronwall's inequalities \cite{Vidar1986Galerkin}:

\textbf{Continuous Gronwall's inequality.} Let $\phi(\cdot)$ be such that
\begin{align*}
	\phi_t(t)\leq\rho\phi(t)+\eta(t) \quad \mathrm{for} \ 0\leq t\leq T,
\end{align*}
where $\rho\geq0$ is a constant and $\eta(\cdot)\geq0,\eta\in L^1([0,T])$. Then it holds 
\begin{equation}
	\phi(t)\leq e^{\rho T}(\phi(0)+\int_{0}^{T}\eta(s)\mathrm{d}s), \quad \forall t\in[0,T].
\end{equation}

\textbf{Discrete Gronwall's inequality.} Let  $g_0\geq0$ and two nonnegative sequences $(k_n)_{n\geq0}$, $(p_n)_{n\geq0}$  be given. Suppose that the sequence $(\phi_n)_{n\geq0}$ satisfies
\begin{equation}
	\left\{
	\begin{array}{ll}
		\phi_0\leq g_0, \\ 
		\phi_n\leq g_0+\sum_{s=0}^{n-1}p_s+\sum_{s=0}^{n-1}k_s\phi_s,\quad \forall n\geq1.
	\end{array}
	\right.
\end{equation}
Then it holds 
\begin{equation}
	\phi_n\leq(g_0+\sum_{s=0}^{n-1}p_s)\exp(\sum_{s=0}^{n-1}k_s),\quad \forall n\geq1.
\end{equation}


We are now in a position to give the weak form of the  Maxwell model \cref{1.1}.  By introducing the velocity variable   $\boldsymbol{v}=\boldsymbol{u}_t$, the model changes into a velocity-stress form:
\begin{equation}\label{1.1-new}
	\left\{
	\begin{array}{ll}
		\rho \boldsymbol{v}_{t}=\mathrm{\textbf{div}}\sigma+\boldsymbol{f}(x,t), & (\boldsymbol{x},t)\in\Omega\times[0,T], \\
		\sigma+\sigma_t=\mathbb{C}\varepsilon(\boldsymbol{v}), &(\boldsymbol{x},t)\in\Omega\times[0,T],\\
		\boldsymbol{v}=0, &  (\boldsymbol{x},t)\in\partial\Omega\times[0,T], \\
		\boldsymbol{v}(\boldsymbol{x},0)=\phi_1(\boldsymbol{x}),\sigma(\boldsymbol{x},0)=\psi_0(\boldsymbol{x}),& \boldsymbol{x}\in\Omega.
	\end{array}
	\right.
\end{equation}
Based on the Hellinger-Reissner variational principle,   the weak problem for  \cref{1.1} reads as: Find $( \sigma,\boldsymbol{v})\in C^1([0,T],\uuline{\mathrm{H}}(\mathrm{\textbf{div}},\Omega,S))\times C^1([0,T],\uline{L}^2(\Omega))$ such that
\begin{equation}\label{weak problem}
	\left\{
	\begin{array}{lll}
		a(\sigma_t,\tau)+a(\sigma,\tau)+b(\boldsymbol{v},\tau)=0, & \forall \tau\in\uuline{\mathrm{H}}(\mathrm{\textbf{div}},\Omega,S), \\ 
		c(\boldsymbol{v}_t,\boldsymbol{w})-b(\boldsymbol{w},\sigma)=(\boldsymbol{f},\boldsymbol{w}), &  \forall \boldsymbol{w}\in\uline{L}^2(\Omega), \\
		\boldsymbol{v}(\boldsymbol{x},0)=\phi_1, \sigma(\boldsymbol{x},0)=\psi_0.
	\end{array}
	\right.
\end{equation}
Here
\begin{equation}\nonumber
		a(\sigma,\tau):=\int_{\Omega}\mathbb{C}^{-1}\sigma:\tau\mathrm{d}\boldsymbol{x}, \quad  
		b(\boldsymbol{v},\tau):=\int_{\Omega}\boldsymbol{v}\cdot \mathrm{\textbf{div}}\tau\mathrm{d}\boldsymbol{x}, \quad 
		c(\boldsymbol{v},\boldsymbol{w}):=\int_{\Omega}\rho \boldsymbol{v}\cdot \boldsymbol{w}\mathrm{d}\boldsymbol{x}, 
\end{equation}
where $\tilde\tau:\tau=\sum\limits_{i,j=1}^d\tilde\tau_{ij}\tau_{ij}$ for $ \tilde\tau, \tau\in \uuline{\mathrm{H}}(\mathrm{\textbf{div}},\Omega,S)$.


For  any $  \tau\in\uuline{\mathrm{H}}(\mathrm{\textbf{div}},\Omega,S),\ \boldsymbol{w}\in\uline{L}^2(\Omega)$, define
\begin{align*}
	||\tau||_a^2:=a(\tau,\tau),\quad ||\boldsymbol{w}||_c^2:=c(\boldsymbol{w},\boldsymbol{w}). 
\end{align*}
Then, according to \cref{2.4} and \cref{2.5},  it holds
\begin{align}
	\sqrt{M_{0}}||\tau||_{0}\leq&||\tau||_a\leq\sqrt{M_{1}}||\tau||_{0}, \quad
	\sqrt{\rho_0}||\boldsymbol{w}||_{0}\leq ||\boldsymbol{w}||_c\leq\sqrt{\rho_1}||\boldsymbol{w}||_{0},
\end{align}
which also give 
\begin{align}\label{a-norm-equi}
	M_0||\tau||_0^2\leq a(\tau,\tau),\quad  \rho_0||\boldsymbol{w}||_0^2\leq c(\boldsymbol{w},\boldsymbol{w}). 
\end{align}
Simultaneously, the following stability conditions hold (\cite{1991Mixed}):

\noindent(i)   Coercivity of   $a(\cdot,\cdot)$ on $Z:=\{ \tau\in\uuline{\mathrm{H}}(\mathrm{\textbf{div}},\Omega,S);\ b(\boldsymbol{v},\tau)=0,\ \forall \boldsymbol{v}\in\underline{L}^2(\Omega) \}$:
\begin{align*}
	 ||\tau||_{\mathrm{\textbf{div}}}^2\lesssim a(\tau,\tau) \quad \forall\tau\in Z,
\end{align*}
where $ ||\tau||_{\mathrm{\textbf{div}}}^2:=||\tau||_{0}^2+||\mathrm{\textbf{div}}\tau||_{0}^2.$

\noindent(ii)   Inf-sup condition for $b(\cdot,\cdot)$:    
\begin{align*}
	||\boldsymbol{w}||_0\lesssim \sup\limits_{0\neq\tau\in\uuline{\mathrm{H}}(\mathrm{\textbf{div}},\Omega,S)} \frac{b(\boldsymbol{w},\tau)}{ ||\tau||_{\mathrm{\textbf{div}}}} \quad \forall \boldsymbol{w}\in\underline{L}^2(\Omega).
\end{align*}

From \cite[Theorem 5.1]{Lee2012Mixed},   the following result of existence and uniqueness holds.
\begin{lemma}
	Suppose $\phi_0\in\underline{H}_0^1(\Omega),\ \phi_1\in\underline{L}^2(\Omega),\ \sigma_0\in\uuline{\mathrm{H}}(\mathrm{\bf{div}},\Omega, S)$ and $\boldsymbol{f}\in C^0([0,T],\underline{L}^2(\Omega))$, then the weak problem \cref{weak problem} admits a unique solution $( \sigma,\boldsymbol{v})\in C^1([0,T],\uuline{\mathrm{H}}(\mathrm{\bf{div}},\Omega,S))\times C^1([0,T],\uline{L}^2(\Omega))$.
\end{lemma}

\section{Semi-discrete   mixed finite element method}
In this section, we discuss the semi-discrete finite element discretization of  \cref{weak problem}   and analyze its convergence under a general conforming mixed FEM framework. 

\subsection{Semi-discrete scheme}
Assume  that $\Omega$ is a convex polyhedral domain, and  let $\mathcal{T}_h=\bigcup\{K\}$ be a shape regular partition of  $\Omega$ consisting of triangles/tetrahedrons or rectangles/cuboids.  For any $K\in \mathcal{T}_h$, let $h_K$ denote its diameter, and we set $h:=\mathop{\mathrm{max}}_{K\in\mathcal{T}_h} h_K$. For any integer $k\geq0$,   let $P_k(K)$ denote the set of all polynomials  on $K$  of   degree at most $k$,   and let $Q_k(K)$ denote the set of all polynomials on $K$ of degree at most $k$ in each variable.


Let $\uuline{\mathrm{H}_h}\subset\uuline{\mathrm{H}}(\mathrm{\textbf{div}},\Omega) $ and $ \underline{\mathrm{V}_h}\subset\uline{L}^2(\Omega)$ be two finite-dimensional spaces    respectively for stress and velocity approximations on  $\mathcal{T}_h$,  satisfying
%
the following   condition:

(A1)   Discrete   inf-sup  condition:
\begin{align*}
	||\boldsymbol{w}_h||_0\lesssim \sup\limits_{0\neq\tau_h\in\uuline{\mathrm{H}_h}} \frac{b(\boldsymbol{w}_h,\tau_h)}{ ||\tau_h||_{\mathrm{\textbf{div}}}} \quad \forall \boldsymbol{w}_h\in\uline{\mathrm{V}_h}.
\end{align*}


 From \cref{a-norm-equi} we easily obtain the following two inequalities: 
\begin{equation}\label{dis-a-norm-equi}
	\begin{array}{ll}
			M_0||\tau_h||_0^2\leq a(\tau_h,\tau_h),\quad &\forall\tau_h\in\uuline{\mathrm{H}_h},\\ 
		\rho_0||\boldsymbol{w}_h||_0^2\leq c(\boldsymbol{w}_h,\boldsymbol{w}_h),\quad &\forall \boldsymbol{w}_h\in\uline{\mathrm{V}_h}.
	\end{array}
\end{equation}

Let $\phi_{1,h},\ \psi_{0,h}$ be respectively approximations of initial data $\phi_1$ and $\psi_0$, then the generic semi-discrete mixed conforming finite element scheme reads as: Find  $( \sigma_h,\boldsymbol{v}_h)\in   C^1([0,T],\uuline{\mathrm{H}_h})\times C^1([0,T],\uline{\mathrm{V}_h})$ such that
\begin{equation}\label{semi-disc}
	\left\{
	\begin{array}{lll}
		a(\sigma_{h,t},\tau_h)+a(\sigma_h,\tau_h)+b(\boldsymbol{v}_h,\tau_h)=0, & \forall \tau_h\in\uuline{\mathrm{H}_h}, \\ 
		c(\boldsymbol{v}_{h,t},\boldsymbol{w}_h)-b(\boldsymbol{w}_h,\sigma_h)=(\boldsymbol{f},\boldsymbol{w}_h), &  \forall \boldsymbol{w}_h\in\uline{\mathrm{V}_h}, \\
		\boldsymbol{v}_h(\boldsymbol{x},0)=\phi_{1,h}, \sigma_h(\boldsymbol{x},0)=\psi_{0,h}.
	\end{array}
	\right.
\end{equation}
By using standard techniques, we can obtain the following conclusion.

\begin{lemma}
	The semi-discrete scheme \cref{semi-disc} admits a unique solution  $( \sigma_h,\boldsymbol{v}_h).$
\end{lemma}
\begin{proof}
	Let $\{\varphi_i\}_{i=1}^{r}$, $\{\boldsymbol{\kappa}_i\}_{i=1}^{s}$ be bases of $\uuline{\mathrm{H}_h}$ and $\uline{\mathrm{V}_h}$ respectively. Let $(i,j)$-th components of matrix $A,\ B,\ C$ be
	\begin{align*}
		(\mathbb{C}^{-1}\varphi_j,\varphi_i),\ (\mathrm{\textbf{div}}\varphi_j,\boldsymbol{\kappa}_i),\ (\boldsymbol{\kappa}_j,\boldsymbol{\kappa}_i),
	\end{align*}
	respectively.
We write $\sigma_h=\sum_{i}\alpha_i\varphi_i,\ \boldsymbol{v}_h=\sum_i\beta_i\boldsymbol{\kappa}_i,\ \eta_j=(\boldsymbol{f},\boldsymbol{\kappa}_j)$, and denote by $\alpha,\ \beta,\ \eta$  the corresponding vectors, respectively.
Then we   rewrite \cref{semi-disc} as a matrix equation of the form
\begin{equation}\label{semimatrix}
		\left(
	\begin{array}{cc}
		A & 0 \\
		0 & C \\
	\end{array}
	\right)
	\left(
	\begin{array}{c}
		\alpha_t \\
		\beta _t
	\end{array}
	\right)
	=
	\left(
	\begin{array}{cc}
		-A & -B \\
		0 & B^{\mathrm{T}} \\ 
	\end{array}
	\right)
	\left(
	\begin{array}{c}
		\alpha \\
		\beta 
	\end{array}
	\right)
	+
	\left(
	\begin{array}{c}
		0 \\ 
		\eta
	\end{array}
	\right)
\end{equation}
The coefficient matrix on the left side of the equation is nonsingular because $A,\ C$ are symmetric positive definite. Thus, due to the standard theory of ordinary differential equations,  the system \cref{semimatrix},   and also \cref{semi-disc}, admits a unique solution.
\end{proof}

\subsection{Error estimation}
To carry out the error estimation, we need to 
introduce, for  the   solution $(\sigma(t), \boldsymbol{v}(t))\in\uuline{\mathrm{H}}(\mathrm{\textbf{div}},\Omega,S)\times \uline{L}^2(\Omega)$ to the 
weak problem  \cref{weak problem}, an  ``elliptic projection" $(\Pi_1\sigma,\Pi_2\boldsymbol{v})\in \uuline{\mathrm{H}_h}\times \uline{\mathrm{V}_h} $
which are defined as follows:
for $t\in[0,T]$, let $(\Pi_1\sigma,\Pi_2\boldsymbol{v}):=(\hat{\sigma}_h(t),\hat{\boldsymbol{v}}_h(t))\in\uuline{\mathrm{H}_h}\times \uline{\mathrm{V}_h}$ satisfy 
\begin{equation}\label{projection}
	\left\{
	\begin{array}{ll}
		a(\hat{\sigma}_h(t),\tau_h)+b(\hat{\boldsymbol{v}}_h(t),\tau_h)=-a(\partial_t\sigma,\tau_h), & \forall \tau_h\in\uuline{\mathrm{H}_h}, \\ 
		b(\boldsymbol{w}_h,\hat{\sigma}_h(t))=b(\boldsymbol{w}_h,\sigma)=c(\boldsymbol{v}_t,\boldsymbol{w}_h)-(\boldsymbol{f},\boldsymbol{w}_h), &  \forall \boldsymbol{w}_h\in\uline{\mathrm{V}_h}.
	\end{array}
	\right.
\end{equation} 
By (A1) and \cref{dis-a-norm-equi} it is easy to see that the   ``elliptic projection" is well-defined.

To derive convergence rates we also make the following regularity and approximation assumptions:

(A2) Let $(\sigma,\boldsymbol{v})$, the weak solution to   \cref{weak problem}, and its elliptic projection $(\hat{\sigma}_h(t),\hat{\boldsymbol{v}}_h(t))$   satisfy the regularity conditions 
	\begin{equation}\label{regularity-assump}
	\left\{
	\begin{array}{ll}
	\sigma\in L^{\infty}([0,T],\uuline{H}^{m}(\Omega)),&  \sigma_t\in L^2([0,T],\uuline{H}^{m}(\Omega)),\\ 
	 \boldsymbol{v}\in L^{\infty}([0,T],\underline{H}^{m'}(\Omega)),&  \boldsymbol{v}_t\in L^2([0,T],\underline{H}^{m'}(\Omega)),
		\end{array}
	\right.
\end{equation} 
and the approximation conditions
\begin{equation}\label{projectionerror}
	\left\{
	\begin{array}{ll}
		||\hat{\sigma}_h-\sigma||_0+||\hat{\boldsymbol{v}}_h-\boldsymbol{v}||_0\lesssim h^{l}(||\sigma||_{m}+||\boldsymbol{v}||_{m'}),\\
		||\hat{\sigma}_{h,t}-\sigma_t||_0+  ||\hat{\boldsymbol{v}}_{h,t}-\boldsymbol{v}_t||_0\lesssim h^{l}(||\sigma_t||_{m}+||v_t||_{m'}),
		\end{array}
	\right.
\end{equation}
where $l, m, m'\geq0$ are integers, 
and $g_{h,t}:=\partial g_{h}/\partial t$ with $g=\hat{\boldsymbol{v}},   \hat{\sigma}$.
\begin{remark}\label{rem3.1}
In the following, we introduce for   $d=2$ and $3$  several pairs of $\uuline{\mathrm{H}_h}$ and  $\uline{\mathrm{V}_h}$ which satisfy both \cref{projectionerror} and the discrete inf-sup condition (A1). 
\begin{itemize}
\item   Arnold-Winther's triangular elements ($d=2$) \cite{2002Mixed}:
\begin{align*}
	&\uuline{\mathrm{H}_h}=\left\{\tau\in\uuline{\mathrm{H}}(\mathrm{\bf{div}},\Omega,S); \ \tau_{ij}|_K\in P_{k+2}(K),\ {\bf{div}}(\tau|_K)\in P_k(K)^2 \ \forall   K\in\mathcal{T}_h\right\}, \\
	&\uline{\mathrm{V}_h}=\left\{\boldsymbol{w}\in\uline{L}^2(\Omega); \  \boldsymbol{w}|_K\in P_k(K)\ \forall K\in\mathcal{T}_h\right\},
	\end{align*}
where $k\geq 1$. The estimates in \cref{projectionerror} hold with $  l=m=k+1$ and $m'=k+2$.

\item Arnold-Awanou's rectangular elements ($d=2$) \cite{Arnold2005}:
\begin{align*}
	&\uuline{\mathrm{H}_h}=\left\{\tau\in\uuline{\mathrm{H}}(\mathrm{\bf{div}},\Omega, S);\ \tau_{11}|_K\in P_{k+4,k+2}(K), \tau_{12}|_K\in P_{k+3,k+3}(K), \right. \\
&\left. \qquad\qquad\qquad \qquad \qquad   \tau_{22}|_K\in P_{k+2,k+4}(K) 
\ \forall K\in\mathcal{T}_h\right\}, \\
	&\uline{\mathrm{V}_h}=\left\{\boldsymbol{w}\in\uline{L}^2(\Omega); \ w_1|_K\in P_{k+1,k}(K),w_2|_K\in P_{k,k+1}(K) \ \forall K\in\mathcal{T}_h\right\},
\end{align*}
 where $k\geq 1$,  and $P_{r,s}(K)$ denotes the set of polynomials, defined  on $K$, of degree at most $r$ in $x_1$ and of degree at most $s$ in $x_2$ for $\boldsymbol{x}=(x_1,x_2)\in K$. 
 The estimates in \cref{projectionerror} hold with $  l=m=k+1$ and $m'=k+2$.

\item Hu-Zhang's triangular/tetrahedral elements ($d=2,3$) \cite{Hu2015FINITE,Hu2014FINITE}:
\begin{align*}
	&\uuline{\mathrm{H}_h}=\left\{\tau\in\uuline{\mathrm{H}}(\mathrm{\bf{div}},\Omega,S); \ \tau_{ij}|_K\in P_{k+d}(K) \ \forall K\in\mathcal{T}_h  \right\}, \\
	&\uline{\mathrm{V}_h}=\left\{\boldsymbol{w}\in\uline{L}^2(\Omega);\  \boldsymbol{w}|_K\in P_{k+d-1}(K)\ \forall K\in\mathcal{T}_h\right\},
\end{align*}
where $k\geq 2$. The estimates in \cref{projectionerror} hold with $   l=m'=k+d-1  $ and $  m=k+d$.

\item   Hu-Man-Zhang's rectangular/cuboid element ($d=2,3$) \cite{2014Hu}:
\begin{align*}
	&\uuline{\mathrm{H}_h}=\left\{\tau\in\uuline{\mathrm{H}}(\mathrm{\bf{div}},\Omega,S); \ \tau_{ii}|_K\in span\{1,x_i,x_i^2\},\right.\\
	&\qquad\qquad  \tau_{ij}|_K\in span\{1,x_i,x_j,x_ix_j\},\  i\neq j, \ \forall K\in\mathcal{T}_h\big\}, \\
	&\uline{\mathrm{V}_h}=\left\{\boldsymbol{w}\in\uline{L}^2(\Omega);\  w_i|_K\in span\{1,x_i \}\ \forall K\in\mathcal{T}_h\right\}.
\end{align*}
The estimates in \cref{projectionerror} hold with $   l=m'=1  $ and $  m=2$.

\item Nedelec's rectangular/cuboid elements ($d=2,3$) \cite{Eliane2001A,J1986A}: 
\begin{align*}
	&\uuline{\mathrm{H}_h}=\left\{\tau\in\uuline{\mathrm{H}}(\mathrm{\bf{div}},\Omega,S); \   \tau_{ij}|_K\in Q_{k+1}(K) \ \forall K\in\mathcal{T}_h\right\},\\
	&\uline{\mathrm{V}_h}=\left\{\boldsymbol{w}\in\uline{L}^2(\Omega);\ w_i|_K\in Q_k(K) \ \forall K\in\mathcal{T}_h\right\},
\end{align*}
where $k\geq 0$. The estimates in \cref{projectionerror} hold with $   l= k+1  $ and $  m=m'=k+2$.  We mention that in \cite{Eliane2001A}  the degrees of freedom  of Nedelec's rectangular elements  $Q_{k+1}^{\mathrm{\textbf{div}}}-Q_k$ ($k\geq 0$) (cf. \cite{J1986A}) are modified so as to allow mass lumping.   

\end{itemize}

\end{remark}

In what follows, we choose the initial data in \cref{semi-disc} as
\begin{equation}\label{initial}
	 \psi_{0,h}=I_{\mathrm{H}_h}\psi_0,\quad \phi_{1,h}=I_{\mathrm{V}_h}\phi_1,
\end{equation}
where $I_{\mathrm{H}_h}:\ \uuline{\mathrm{H}}(\mathrm{\textbf{div}},\Omega,S)\rightarrow\uuline{\mathrm{H}_h}$, $I_{\mathrm{V}_h}:\uline{L}^2(\Omega)\rightarrow\uline{\mathrm{V}_h}$ be two interpolation operators   satisfying
 \begin{equation}\label{initial-estimate}
 ||\psi_0-I_{\mathrm{H}_h}\psi_0||_0\lesssim h^l||\psi_0||_l, \quad ||\phi_1-I_{\mathrm{V}_h}\phi_1||_0\lesssim h^l||\phi_1||_l
 \end{equation}
 for 	$
	 \psi_0\in \uuline{H}^{l}(\Omega), \quad \phi_1\in \underline{H}^l(\Omega).
	$

\begin{theorem}\label{semi-theorem}
	Let $( \sigma,\boldsymbol{v})\in C^1([0,T],\uuline{\mathrm{H}}(\mathrm{\bf{div}},\Omega,S))\times C^1([0,T],\uline{L}^2(\Omega))$ be the solution of  the weak problem \cref{weak problem} and $(\sigma_h,\boldsymbol{v}_h)\in   C^1([0,T],\uuline{\mathrm{H}_h})\times C^1([0,T],\uline{\mathrm{V}_h})$ be the solution of the semi-discrete problem \cref{semi-disc}. Then,  under the assumptions (A1), (A2)  and \cref{initial-estimate} we have
	\begin{align}\label{semi-estimate}
		&\quad  ||\sigma-\sigma_h||_{\mathrm{L}^\infty([0,T],L^2)}+||\boldsymbol{v}-\boldsymbol{v}_h||_{\mathrm{L}^\infty([0,T],L^2)} \nonumber\\ 
		&\lesssim h^l(||\psi_0||_l+||\phi_1||_l+||\sigma||_{L^{\infty}(H^m)}+||\sigma_t||_{L^2(H^{m})}+||\boldsymbol{v}||_{L^{\infty}(H^{m'})}+||\boldsymbol{v}_t||_{L^2(H^{m'})}).
	\end{align}

\end{theorem}  
 
\begin{proof}
	In light of   \cref{projectionerror}, \cref{projection}  and  the triangle inequality, it suffices to   show the estimate
	\begin{equation}\label{projection-est}
		||\hat{\sigma}_h-\sigma_h||_{0}+||\hat{\boldsymbol{v}}_h-\boldsymbol{v}_h||_{0}\lesssim h^l(||\psi_0||_l+||\phi_1||_l+||\sigma_t||_{L^2(H^{m})}+||\boldsymbol{v}_t||_{L^2(H^{m'})}).
	\end{equation}
	  From \cref{projection} and  \cref{weak problem} it follows 
	\begin{align*}
		&a(\partial_t(\sigma-\sigma_h),\tau_h)+a(\hat{\sigma}_h-\sigma_{h},\tau_h)+b(\hat{\boldsymbol{v}}_h-\boldsymbol{v}_h,\tau_h)=0, \ \forall \tau_h \in\uuline{\mathrm{H}_h} ,\\
		&c(\partial_t(v-v_h),\boldsymbol{w}_h)=b(\boldsymbol{w}_h,\hat{\sigma}_h-\sigma_h), \ \forall  {\boldsymbol{w}}_h \in  \uline{\mathrm{V}_h}.
	\end{align*}
Denote $\tilde\sigma_h:=\hat{\sigma}_h-\sigma_h,\ \tilde {\boldsymbol{v}}_h:=\hat{\boldsymbol{v}}_h-\boldsymbol{v}_h$, and take $\tau_h=\tilde\sigma_h,\ \boldsymbol{w}_h=\tilde{\boldsymbol{v}}_h$ in the above two equations, then we get
\begin{align*}
	a(\partial_t(\sigma-\sigma_h),\tilde\sigma_h)+a(\tilde\sigma_h,\tilde\sigma_h)+c(\partial_t(\boldsymbol{v}-\boldsymbol{v}_h),\tilde{\boldsymbol{v}}_h)=0,
\end{align*}
which yields
\begin{align*}
	a(\partial_t(\sigma-\sigma_h),\tilde\sigma_h)+c(\partial_t(\boldsymbol{v}-\boldsymbol{v}_h),\tilde{\boldsymbol{v}}_h)\leq0.
\end{align*}
This, together with the relations $\sigma-\sigma_h=\sigma-\hat{\sigma}_h+\tilde\sigma_h$ and $ \boldsymbol{v}-\boldsymbol{v}_h=\boldsymbol{v}-\hat{\boldsymbol{v}}_h+\tilde{\boldsymbol{v}}_h$, implies
\begin{align*}
	a(\partial_t\tilde\sigma_h,\tilde\sigma_h)+c(\partial_t \tilde{\boldsymbol{v}}_h,\tilde{\boldsymbol{v}}_h)\leq a(\partial_t(\hat{\sigma}_h-\sigma),\tilde\sigma_h)+c(\partial_t(\hat{\boldsymbol{v}}_h-\boldsymbol{v}),\tilde{\boldsymbol{v}}_h),
\end{align*}
Thus, by setting $E_h:=a(\tilde\sigma_h,\tilde\sigma_h)+c(\tilde{\boldsymbol{v}}_h,\tilde{\boldsymbol{v}}_h) $
 we can obtain
\begin{eqnarray*}
	\frac{\mathrm{d}E_h}{\mathrm{d}t}&\leq& 2a(\partial_t(\hat{\sigma}_h-\sigma),\tilde\sigma_h)+2c(\partial_t(\hat{\boldsymbol{v}}_h-\boldsymbol{v}),\tilde{\boldsymbol{v}}_h)\\
	&\leq& E_h+a(\partial_t(\hat{\sigma}_h-\sigma),\partial_t(\hat{\sigma}_h-\sigma))+c(\partial_t(\hat{\boldsymbol{v}}_h-\boldsymbol{v}),\partial_t(\hat{\boldsymbol{v}}_h-\boldsymbol{v})),
\end{eqnarray*}
where we have used the  following two inequalities: 
$$2a(\sigma,\tau)\leq a(\sigma,\sigma)+a(\tau,\tau), \ \   2c(\boldsymbol{v},\boldsymbol{w})\leq c(\boldsymbol{v},\boldsymbol{v})+c(\boldsymbol{w},\boldsymbol{w}).$$

By the continuous Gronwall inequality we deduce that 
\begin{align*}
	&\quad E_h=||\tilde\sigma_h||_a^2+||\tilde{\boldsymbol{v}}_h||_c^2\\
	&\lesssim||\tilde\sigma_h(0)||_a^2+||\tilde{\boldsymbol{v}}_h(0)||_c^2 
	+\int_{0}^{T}||\partial_t(\hat{\sigma}_h(s)-\sigma(s))||_a^2+||\partial_t(\hat{\boldsymbol{v}}_h(s)-\boldsymbol{v}(s))||_c^2\mathrm{d}s
\end{align*}
As a result, the desired estimate \cref{projection-est} follows from the initial data condition \cref{initial-estimate}, the assumption (A2) and the equivalence of norms in \cref{a-norm-equi}.
 
\end{proof}

\section{Fully discrete mixed finite element method}
\subsection{Fully discrete scheme}

Let $0=t_0<t_1<...<t_{\mathrm{M}}=T$ be a uniform division of time domain $[0,T]$,  with $t_i=i\Delta t \ (i=0,1,...,M)$,   and  the time step size $\Delta t:=\frac{T}{\mathrm{M}}$. 
For any function $\varphi(t)$, we set
\begin{align*}
	\varphi^n:=\varphi(t_n),\ \varphi^{n+\frac{1}{2}}:=\frac{\varphi^n+\varphi^{n+1}}{2},\
	\Delta_t\varphi^{n+\frac{1}{2}}:=\frac{\varphi^{n+1}-\varphi^n}{\Delta t}. 
\end{align*}
Based on the semi-discrete scheme
\cref{semi-disc}, a   Crank-Nicolson full discretization scheme is given as follows: 
Find  $(\sigma_h^{n+1},\boldsymbol{v}_h^{n+1})\in \uuline{\mathrm{H}_h}\times \uline{\mathrm{V}_h}$ for $ 0\leq n\leq M-1 $ such that
\begin{equation}\label{full}
		\left\{
	\begin{array}{ll}
		a(\Delta_t\sigma_h^{n+\frac{1}{2}},\tau_h)+a(\sigma_h^{n+\frac{1}{2}},\tau_h)+b(\boldsymbol{v}_h^{n+\frac{1}{2}},\tau_h)=0, & \forall \tau_h\in\uuline{\mathrm{H}_h}, \\ 
		c(\Delta_t\boldsymbol{v}_h^{n+\frac{1}{2}},\boldsymbol{w}_h)-b(\boldsymbol{w}_h,\sigma_h^{n+\frac{1}{2}})=(\boldsymbol{f}^{n+\frac{1}{2}},\boldsymbol{w}_h), &  \forall \boldsymbol{w}_h\in\uline{\mathrm{V}_h},
	\end{array}
	\right.
\end{equation}
with the initial data $\boldsymbol{v}_h^0=\phi_{1,h}$ and $ \sigma_h^0=\psi_{0,h}$ given by \cref{initial}.  

\begin{lemma}
	The fully discrete scheme \cref{full} admits a unique solution $(\sigma_h^{n+1},\boldsymbol{v}_h^{n+1})$ for $n=0,1,...,M-1.$
\end{lemma} 
\begin{proof}
	We only need  to show that, when given $(\sigma_h^{n},\boldsymbol{v}_h^{n})$,  the linear system \cref{full} admits a unique solution $(\sigma_h^{n+1},\boldsymbol{v}_h^{n+1})$. Since this is a  square system, it suffices to show that  the homogeneous system 
	\begin{equation*} 
		\left\{
	\begin{array}{ll}
		a(\sigma_h^{n+1},\tau_h)+\frac{\Delta t}{2}a(\sigma_h^{n+1},\tau_h)+\frac{\Delta t}{2}b(\boldsymbol{v}_h^{n+1},\tau_h)
		=0& \forall \tau_h\in\uuline{\mathrm{H}_h},\\
		c(\boldsymbol{v}_h^{n+1},\boldsymbol{w}_h)-\frac{\Delta t}{2}b(\boldsymbol{w}_h,\sigma_h^{n+1})
		=0 &\forall \boldsymbol{w}_h\in\uline{\mathrm{V}_h}
	\end{array}
	\right.
\end{equation*}
yields the conclusion that 
\begin{equation}\label{zero}
\sigma_h^{n+1}=0,\quad \boldsymbol{v}_h^{n+1}=0.
\end{equation}   
In fact, taking $\tau_h=\sigma_h^{n+1}$ and $\boldsymbol{w}_h=\boldsymbol{v}_h$ in the above system leads to 
 $$(1+\frac{\Delta t}{2})a(\sigma_h^{n+1},\sigma_h^{n+1})+c(\boldsymbol{v}_h^{n+1},\boldsymbol{v}_h^{n+1})=0,$$ 
 then \cref{zero} follows. This completes the proof.
\end{proof}

\subsection{Stability analysis}

\begin{lemma}
	For $J=1,2,...,M,$ it holds 
	\begin{equation}\label{stable}
		||\sigma_h^J||_a^2+||\boldsymbol{v}_h^J||_c^2\leq2\Delta t\sum_{n=0}^{J-1}(\boldsymbol{f}^{n+\frac{1}{2}},\boldsymbol{v}_h^{n+\frac{1}{2}})+||\sigma_h^0||_a^2+||\boldsymbol{v}_h^0||_c^2.
	\end{equation}
\end{lemma}

\begin{proof}
	Take $\tau_h=\sigma_h^{n+\frac{1}{2}}$ and $\boldsymbol{w}_h=\boldsymbol{v}_h^{n+\frac{1}{2}}$ in \cref{full} and add up the two equations,  we then get
	\begin{align*}
		a(\Delta_t\sigma_h^{n+\frac{1}{2}},\sigma_h^{n+\frac{1}{2}})+a(\sigma_h^{n+\frac{1}{2}},\sigma_h^{n+\frac{1}{2}})+c(\Delta_t\boldsymbol{v}_h^{n+\frac{1}{2}},\boldsymbol{v}_h^{n+\frac{1}{2}})=(\boldsymbol{f}^{n+\frac{1}{2}},\boldsymbol{v}_h^{n+\frac{1}{2}}).
	\end{align*}
By the symmetry of $\mathbb{C}^{-1}$, we   deduce that
\begin{align*}
	a(\Delta_t\sigma_h^{n+\frac{1}{2}},\sigma_h^{n+\frac{1}{2}})=(\mathbb{C}^{-1}\Delta_t\sigma_h^{n+\frac{1}{2}},\sigma_h^{n+\frac{1}{2}})&=(\mathbb{C}^{-1}\frac{\sigma_h^{n+1}-\sigma_h^n}{\Delta t},\frac{\sigma_h^{n+1}+\sigma_h^n}{2})\\
	&=\frac{1}{2\Delta t}(||\sigma_h^{n+1}||_a^2-||\sigma_h^{n}||_a^2).
	\end{align*}
Similiarity, we have $$c(\Delta_t\boldsymbol{v}_h^{n+\frac{1}{2}},\boldsymbol{v}_h^{n+\frac{1}{2}})=\frac{1}{2\Delta t}(||\boldsymbol{v}_h^{n+1}||_c^2-||\boldsymbol{v}_h^{n}||_c^2).$$
 From the  two relations above it follows that 
\begin{align*}
	\frac{1}{2\Delta t}(||\sigma_h^{n+1}||_a^2-||\sigma_h^{n}||_a^2+||\boldsymbol{v}_h^{n+1}||_c^2-||\boldsymbol{v}_h^{n}||_c^2)+||\sigma_h^{n+\frac{1}{2}}||_a^2=(\boldsymbol{f}^{n+\frac{1}{2}},\boldsymbol{v}_h^{n+\frac{1}{2}}).
\end{align*}
Summing up this equation from $n=0,1,...,J-1$ gives 
\begin{align*}
	||\sigma_h^J||_a^2+||\boldsymbol{v}_h^J||_c^2+2\Delta t\sum_{n=0}^{J-1}||\sigma_h^{n+\frac{1}{2}}||_a^2=2\Delta t\sum_{n=0}^{J-1}(\boldsymbol{f}^{n+\frac{1}{2}},\boldsymbol{v}_h^{n+\frac{1}{2}})+||\sigma_h^0||_a^2+||\boldsymbol{v}_h^0||_c^2,
\end{align*}
which indicates the desired result.
\end{proof}

\begin{theorem}\label{sta-result}
	Assume $\Delta t<1$, then the full discretization scheme \cref{full} is unconditionally stable in the following sense: For J=1,2,...,M, it holds
\begin{equation}\label{44.4}
	||\boldsymbol{v}_h^J||_{c}^2+||\sigma_h^J||_a^2\lesssim ||\sigma_h^0||_a^2+||\boldsymbol{v}_h^0||_{c}^2+||\boldsymbol{f}||_{L^{\infty}( L^2)}^2. 
\end{equation}
\end{theorem}

\begin{proof}
	From \cref{stable} and Cauchy-Schwarz inequality we get 
	\begin{align*}
		||\sigma_h^J||_a^2+||\boldsymbol{v}_h^J||_c^2\leq\Delta t\sum_{n=0}^{J-1}\rho^{-1}||\boldsymbol{f}^{n+\frac{1}{2}}||_{0}^2+\Delta t\sum_{n=0}^{J-1}||\boldsymbol{v}_h^{n+\frac{1}{2}}||_{c}^2+||\sigma_h^0||_a^2+||\boldsymbol{v}_h^0||_c^2.
	\end{align*}
On the other hand,
\begin{align*}
	\sum_{n=0}^{J-1}||\boldsymbol{v}_h^{n+\frac{1}{2}}||_{c}^2=\sum_{n=0}^{J-1}||\frac{\boldsymbol{v}_h^{n+1}+\boldsymbol{v}_h^n}{2}||_{c}^2&\leq\frac{1}{2}(\sum_{n=0}^{J-1}||\boldsymbol{v}_h^{n+1}||_{c}^2+\sum_{n=0}^{J-1}||\boldsymbol{v}_h^{n}||_{c}^2) \\
	&\leq\sum_{n=0}^{J-1}||\boldsymbol{v}_h^{n}||_{c}^2+\frac{1}{2}||\boldsymbol{v}_h^J||_{c}^2.
\end{align*}
Since $\Delta t<1$, the two inequalities above  indicate
\begin{align*}
	\frac{1}{2}||\boldsymbol{v}_h^J||_{c}^2+||\sigma_h^J||_a^2\leq\Delta t\sum_{n=0}^{J-1}\rho^{-1}||\boldsymbol{f}^{n+\frac{1}{2}}||_{0}^2+\Delta t\sum_{n=0}^{J-1}||\boldsymbol{v}_h^{n}||_{c}^2
	+||\sigma_h^0||_a^2+||\boldsymbol{v}_h^0||_c^2,
\end{align*}
which, together with the discrete Gronwall's inequality, yields \begin{align*}
	||\boldsymbol{v}_h^J||_{c}^2+||\sigma_h^J||_a^2\leq \left(||\sigma_h^0||_a^2+||\boldsymbol{v}_h^0||_{c}^2+T\rho^{-1}||\boldsymbol{f}||_{L^{\infty}( L^2)}^2\right)\times \exp(2T),
\end{align*}
i.e., \cref{44.4} holds true.
\end{proof}

\subsection{Error estimation}

\begin{lemma}\label{lemma4.3}
	Let $(\sigma_h^j,\boldsymbol{v}_h^j)$ $(j=1,2,\cdots,M)$ and $(\hat{\sigma}_h,\hat{\boldsymbol{v}}_h)$ be respectively the solutions to the fully discrete scheme \cref{full} and the semi-discrete "elliptic projection" problem \cref{projection}, then  it holds 
	\begin{equation}\label{4.3}
		\begin{aligned}
			&\quad \mathop{\mathrm{max}}_{1\leq j\leq M}||\sigma_h^{j}-\hat{\sigma}_h(t_{j})||_a+\mathop{\mathrm{max}}_{1\leq j\leq M}||\boldsymbol{v}_h^{j}-\hat{\boldsymbol{v}}_h(t_{j})||_c \\
			&\lesssim ||\sigma_h^{0}-\hat{\sigma}_h(0)||_0+||\boldsymbol{v}_h^{0}-\hat{\boldsymbol{v}}_h(0)||_0\\
			&\quad+
			\Delta t(\sum_{j=0}^{M-1}||\partial_t\sigma^{j+\frac{1}{2}}-\Delta_t\hat{\sigma}_h^{j+\frac{1}{2}}||_{0}+\sum_{j=0}^{M-1}||\partial_t \boldsymbol{v}^{j+\frac{1}{2}}-\Delta_t\hat{\boldsymbol{v}}_h^{j+\frac{1}{2}}||_{0}).
		\end{aligned}
	\end{equation}
\end{lemma}

\begin{proof}
	Setting $\psi_h^{\alpha}:=\sigma_h^{\alpha}-\hat{\sigma}_h(t_{\alpha}),\ \boldsymbol{r}_h^{\alpha}:=\boldsymbol{v}_h^{\alpha}-\hat{\boldsymbol{v}}_h(t_{\alpha})$ for any index $\alpha$ and taking $t=t_j,\ t_{j+1}$ in \eqref{projection} respectively, from \eqref{full} we   have, for $\forall\tau_h\in\uuline{\mathrm{H}_h}$ and $ \boldsymbol{w}_h\in\uline{\mathrm{V}_h}$,
	\begin{equation*}\label{difference}
		\left\{
		\begin{array}{ll}
			a(\psi_h^{j+\frac{1}{2}},\tau_h)+b(\boldsymbol{r}_h^{j+\frac{1}{2}},\tau_h)=-a(\Delta_t\psi_h^{j+\frac{1}{2}},\tau_h)+a(\partial_t\sigma^{j+\frac{1}{2}}-\Delta_t\hat{\sigma}_h^{j+\frac{1}{2}},\tau_h),\\
			b(\boldsymbol{w}_h,\psi_h^{j+\frac{1}{2}})=c(\Delta_t\boldsymbol{r}_h^{j+\frac{1}{2}},\boldsymbol{w}_h)+c(\Delta_t\hat{\boldsymbol{v}}_h^{j+\frac{1}{2}}-\partial_t \boldsymbol{v}^{j+\frac{1}{2}},\boldsymbol{w}_h).
		\end{array}
		\right.
	\end{equation*}
Take $\tau_h=\psi_h^{j+\frac{1}{2}}$ and $\boldsymbol{w}_h=\boldsymbol{r}_h^{j+\frac{1}{2}}$  in these two equations, respectively,   then we obtain
	\begin{align*}
		& a(\psi_h^{j+\frac{1}{2}},\psi_h^{j+\frac{1}{2}})+b(\boldsymbol{r}_h^{j+\frac{1}{2}},\psi_h^{j+\frac{1}{2}})=-a(\Delta_t\psi_h^{j+\frac{1}{2}},\psi_h^{j+\frac{1}{2}}) +a(\partial_t\sigma^{j+\frac{1}{2}}-\Delta_t\hat{\sigma}_h^{j+\frac{1}{2}},\psi_h^{j+\frac{1}{2}}),\\
	&  b(\boldsymbol{r}_h^{j+\frac{1}{2}},\psi_h^{j+\frac{1}{2}})=c(\Delta_t\boldsymbol{r}_h^{j+\frac{1}{2}},\boldsymbol{r}_h^{j+\frac{1}{2}})+c(\Delta_t\hat{\boldsymbol{v}}_h^{j+\frac{1}{2}}-\partial_t\boldsymbol{v}^{j+\frac{1}{2}},\boldsymbol{r}_h^{j+\frac{1}{2}}).
\end{align*}
Subtracting the second one of the above two  equations from the first one, we arrive at 
\begin{equation*}\label{4.3.3}
	\begin{aligned}
		&a(\psi_h^{j+\frac{1}{2}},\psi_h^{j+\frac{1}{2}})+c(\Delta_t\boldsymbol{r}_h^{j+\frac{1}{2}},\boldsymbol{r}_h^{j+\frac{1}{2}})+a(\Delta_t\psi_h^{j+\frac{1}{2}},\psi_h^{j+\frac{1}{2}})\\
		&=a(\partial_t\sigma^{j+\frac{1}{2}}-\Delta_t\hat{\sigma}_h^{j+\frac{1}{2}},\psi_h^{j+\frac{1}{2}})+c(\partial_t \boldsymbol{v}^{j+\frac{1}{2}}-\Delta_t\hat{\boldsymbol{v}}_h^{j+\frac{1}{2}},\boldsymbol{r}_h^{j+\frac{1}{2}}),
	\end{aligned}
\end{equation*} 
which implies
\begin{equation}\label{4.3.4}
	\begin{aligned}
		&\frac{1}{2\Delta t}(||\psi_h^{j+1}||_a^2-||\psi_h^{j}||_a^2+||\boldsymbol{r}_h^{j+1}||_c^2-||\boldsymbol{r}_h^{j}||_c^2) \\
		&\lesssim a(\partial_t\sigma^{j+\frac{1}{2}}-\Delta_t\hat{\sigma}_h^{j+\frac{1}{2}},\psi_h^{j+\frac{1}{2}})+c(\partial_t \boldsymbol{v}^{j+\frac{1}{2}}-\Delta_t\hat{\boldsymbol{v}}_h^{j+\frac{1}{2}},\boldsymbol{r}_h^{j+\frac{1}{2}}),
	\end{aligned}
\end{equation}
Multipling this inequality by $2\Delta t, $  and summing these equations for $j=0,1,\cdots,n-1$ ($1\leq n\leq M$), we get
\begin{equation}\label{4.3.5}
		||\psi_h^n||_a^2+||\boldsymbol{r}_h^n||_c^2\lesssim ||\psi_h^0||_a^2+||\boldsymbol{r}_h^0||_c^2+A_1 +A_2,
\end{equation}
with $$A_1:=\Delta t\sum_{j=0}^{n-1}a(\partial_t\sigma^{j+\frac{1}{2}}-\Delta_t\hat{\sigma}_h^{j+\frac{1}{2}},\psi_h^{j+1}+\psi_h^j),$$
	$$A_2:= \Delta t\sum_{j=0}^{n-1}c(\partial_t \boldsymbol{v}^{j+\frac{1}{2}}-\Delta_t\hat{\boldsymbol{v}}_h^{j+\frac{1}{2}},\boldsymbol{r}_h^{j+1}+\boldsymbol{r}_h^j).$$
For the term $A_1$,  it holds 
\begin{equation}\label{4.3.6}
	\begin{aligned}
		&A_1\leq \tilde{C} \Delta t \sum_{j=0}^{n-1}||\partial_t\sigma^{j+\frac{1}{2}}-\Delta_t\hat{\sigma}_h^{j+\frac{1}{2}}||_{0}||\psi_h^{j+1}+\psi_h^j||_{0} \\
		&\quad \leq 2 \tilde{C} \Delta t\mathop{\mathrm{max}}_{0\leq j\leq n}||\psi_h^{j}||_{0}(\sum_{j=0}^{n-1}||\partial_t\sigma^{j+\frac{1}{2}}-\Delta_t\hat{\sigma}_h^{j+\frac{1}{2}}||_{0})  \\
		&\quad \leq\frac{1}{2}\mathop{\mathrm{max}}_{0\leq j\leq n }||\psi_h^{j }||_{0}^2+2(\tilde{C}\Delta t)^2(\sum_{j=0}^{n-1}||\partial_t\sigma^{j+\frac{1}{2}}-\triangle_t\hat{\sigma}_h^{j+\frac{1}{2}}||_{0})^2,
	\end{aligned}
\end{equation}
where   $\tilde{C}>0$ is a constant depending on $\mathbb{C}^{-1}$ and $\rho$.  Similarly, we have 
\begin{equation}\label{4.3.7}
	\begin{aligned}
		A_2\leq\frac{1}{2}\mathop{\mathrm{max}}_{0\leq j\leq n }||\boldsymbol{r}_h^{j }||_{0}^2+2(\rho\Delta t)^2(\sum_{j=0}^{n-1}||\partial_t \boldsymbol{v}^{j+\frac{1}{2}}-\Delta_t\hat{\boldsymbol{v}}_h^{j+\frac{1}{2}}||_{0})^2.
	\end{aligned}
\end{equation}
Putting \eqref{4.3.7} and \eqref{4.3.6} into \eqref{4.3.5} and noticing that $1\leq n\leq M$, we finally get 
\begin{equation*}\label{4.3.8}
	\begin{aligned}
		&\quad \mathop{\mathrm{max}}_{1\leq j\leq M}||\psi_h^{j}||_a^2+\mathop{\mathrm{max}}_{1\leq j\leq M}||\boldsymbol{r}_h^{j}||_c^2 \\
		&\lesssim||\psi_h^0||_0^2+||\boldsymbol{r}_h^0||_0^2+( \Delta t)^2(\sum_{j=0}^{M-1}||\partial_t\sigma^{j+\frac{1}{2}}-\Delta_t\hat{\sigma}_h^{j+\frac{1}{2}}||_{0})^2 \\
		&\quad +( \Delta t)^2(\sum_{j=0}^{M-1}||\partial_t \boldsymbol{v}^{j+\frac{1}{2}}-\Delta_t\hat{\boldsymbol{v}}_h^{j+\frac{1}{2}}||_{0})^2,
	\end{aligned}	
\end{equation*}
i.e. \cref{4.3} holds true.
\end{proof}

\begin{lemma}\label{lemma4.4}
	Under the assumption (A2)  and the condition  that
	\begin{equation}\label{tt-reg}
	\sigma_{tt}\in{\mathrm{L}^{\infty}([0,T],\uuline{L}^2(\Omega))},\quad \boldsymbol{v}_{tt}\in{\mathrm{L}^{\infty}([0,T],\underline{L}^2(\Omega))},
	\end{equation}
	 it holds,  for $0\leq j\leq M-1$, 
	\begin{align*}
		&\quad \Delta t(||\partial_t\sigma^{j+\frac{1}{2}}-\Delta_t\hat{\sigma}_h^{j+\frac{1}{2}}||_{0}+||\partial_t\boldsymbol{v}^{j+\frac{1}{2}}-\Delta_t\hat{\boldsymbol{v}}_h^{j+\frac{1}{2}}||_{0})  \\
		&\lesssim h^l (||\sigma||_{L^{\infty}(H^{m})}+||\boldsymbol{v}||_{L^{\infty}(H^{m'})})+(\Delta t)^2(||\boldsymbol{v}_{tt}||_{\mathrm{L}^{\infty}(L^2)}+||\sigma_{tt}||_{\mathrm{L}^{\infty}(L^2)}).
	\end{align*}
\end{lemma}

\begin{proof} 
On one hand, using the Taylor expansion, we have
\begin{align*}
	\partial_t\sigma^{j+\frac{1}{2}}-\Delta_t\sigma^{j+\frac{1}{2}}=\frac{1}{2}\int_{t_j}^{t_{j+1}}\sigma_{tt}(s)\mathrm{d}s-\frac{1}{\Delta t}\int_{t_j}^{t_{j+1}}(t_{j+1}-s)\sigma_{tt}(s)\mathrm{d}s,
\end{align*}
which gives 
 $$||\partial_t\sigma^{j+\frac{1}{2}}-\Delta_t\sigma^{j+\frac{1}{2}}||_{0}\lesssim \Delta t ||\sigma_{tt}||_{\mathrm{L}^{\infty}(L^2)}.$$
On the other hand, from  \cref{projectionerror} it follows
\begin{align*}
	\Delta t ||\Delta_t\sigma^{j+\frac{1}{2}}-\Delta_t\hat{\sigma}_h^{j+\frac{1}{2}}||_{0}&=|| {\sigma(t_{j+1})-\hat{\sigma}_h(t_{j+1})-(\sigma(t_j)-\hat{\sigma}(t_j))} ||_{0} \\
	&\lesssim {h^{l}}(||\sigma||_{L^{\infty}(H^{m})}+||\boldsymbol{v}||_{L^{\infty}(H^{m'})})  .
\end{align*}
As a result,  by the triangle inequality we get
%
%
%
\begin{align*}
	\Delta t||\partial_t\sigma^{j+\frac{1}{2}}-\Delta_t\hat{\sigma}_h^{j+\frac{1}{2}}||_{0}\lesssim h^{l}(||\sigma||_{L^{\infty}(H^{m})}+||\boldsymbol{v}||_{L^{\infty}(H^{m'})})+(\Delta t)^2||\sigma_{tt}||_{\mathrm{L}^{\infty}(L^2)}.
\end{align*}
In the same way, we can obtain
\begin{align*}
	\Delta t||\partial_t\boldsymbol{v}^{j+\frac{1}{2}}-\Delta_t\hat{\boldsymbol{v}}_h^{j+\frac{1}{2}}||_{0}\lesssim h^{l}(||\sigma||_{L^{\infty}(H^{m})}+||\boldsymbol{v}||_{L^{\infty}(H^{m'})})+(\Delta t)^2||\boldsymbol{v}_{tt}||_{\mathrm{L}^{\infty}(L^2)}.
\end{align*}
This finishes the proof.
\end{proof}
%

Based on  \cref{projectionerror}, \cref{initial-estimate}, and  \cref{lemma4.3,lemma4.4}, it is easy to give the following error estimate for the   fully discrete finite element scheme.

\begin{theorem}\label{theorem4.2}
	Let $(\sigma(t),\boldsymbol{v}(t))$  be the solution to the weak problem \cref{weak problem} and $(\sigma_h^{n},\boldsymbol{v}_h^{n})\ (n= 1,...,M)$ be the solution to the fully discrete scheme \cref{full} such that the assumptions (A1), (A2),  \cref{initial-estimate,tt-reg}  hold. 
	Then it holds the   error estimate
	\begin{equation}\label{main-est}
		\begin{aligned}
			&\mathop{\mathrm{max}}_{1\leq n\leq M}||\sigma(t_{n})-\sigma_h(t_{n })||_a+\mathop{\mathrm{max}}_{1\leq n\leq M}||\boldsymbol{v}(t_{n})-\boldsymbol{v}_h(t_{n})||_c \lesssim C_1 h^l+C_2(\Delta t)^2,
		\end{aligned}
	\end{equation}
where $C_1:= ||\sigma||_{L^{\infty}(H^{m})}+||\boldsymbol{v}||_{L^{\infty}(H^{m'})}+||\psi_0||_l+||\phi_1||_l$ and $ C_2:=||\sigma_{tt}||_{\mathrm{L}^{\infty}(L^2)}+||\boldsymbol{v}_{tt}||_{\mathrm{L}^{\infty}(L^2)}.$
\end{theorem}

\begin{remark}\label{rem4.1}
From \cref{rem3.1,theorem4.2} we easily see that the error estimate \cref{main-est} holds for 
\begin{itemize}
\item    Arnold-Winther's triangular elements ($d=2$) \cite{2002Mixed} and  Arnold-Awanou's rectangular elements ($d=2$) \cite{Arnold2005} with   $  l=m=k+1$ $ (k\geq 1)$ and $m'=k+2$;
\item  Hu-Zhang's triangular/cuboid elements ($d=2,3$) \cite{Hu2015FINITE,Hu2014FINITE} with $   l=m'=k+d-1  $ ($k\geq 2$) and $  m=k+d$;
 \item 
  Hu-Man-Zhang's rectangular/cuboid element ($d=2,3$) \cite{2014Hu} with $   l=m'=1  $ and $  m=2$;  
\item Nedelec's rectangular/cuboid elements ($d=2,3$) \cite{Eliane2001A,J1986A} with $   l= k+1  $ ($k\geq 0$) and $  m=m'=k+2$.
\end{itemize}



\end{remark}

\section{  Numerical results}

As shown in \cref{rem4.1}, there are many existing mixed conforming finite elements that can be used in the discretization of the two- or three-dimensional Maxwell viscoelastic model \cref{1.1}. 
In this section, we only consider  two-dimensional    numerical examples (Examples 5.1-5.3)  and apply  the following two low order rectangular elements
in the full discretization scheme  \cref{full}:
\begin{itemize}
\item The lowest order   modified Nedelec's rectangular element  \cite{Eliane2001A} ($k=0$) with mass lumping: $Q_{1}^{\mathrm{\textbf{div}}}-Q_0$ element.
The corresponding finite-dimensional spaces are 
\begin{align*}
	&\uuline{\mathrm{H}_h}=\left\{\tau\in\uuline{\mathrm{H}}(\mathrm{\bf{div}},\Omega,S); \   \tau_{ij}|_K\in Q_{1}(K) \ \forall K\in\mathcal{T}_h\right\},\\
	&\uline{\mathrm{V}_h}=\left\{\boldsymbol{w}\in\uline{L}^2(\Omega);\ w_i|_K\in Q_0(K) \ \forall K\in\mathcal{T}_h\right\},
\end{align*}
and the  local degrees of freedom for the stress tensor $\sigma_h\in \uuline{\mathrm{H}_h}$ are $\sigma_h(T_i)$ ($i=1,2,3,4$), i.e., the  values of $\sigma_h$ at the four vertices of  rectangular element $K$.
In the computation of $a(\cdot,\cdot)$, 
 the following quadrature formula on $K$ is  used for mass lumping \cite{Eliane2001A}:
\begin{align*}
	\int_{K}g\mathrm{d}x\approx\frac{h_xh_y}{4}\sum_{i=1}^4g(T_i) \qquad\forall g\in C^0(K),
\end{align*}
where $h_x$ and $h_y$ are the side lengths of $K$. 

\item  Hu-Man-Zhang's (abbr. HMZ) rectangular element \cite{2014Hu}.  We recall in this case that
\begin{align*}
	&\uuline{\mathrm{H}_h}=\left\{\tau\in\uuline{\mathrm{H}}(\mathrm{\bf{div}},\Omega,S); \tau_{ii}|_K\in span\{1,x_i,x_i^2\}, 
	\tau_{12}|_K\in Q_1(K)\  \forall K\in\mathcal{T}_h \right\}, \\
	&\uline{\mathrm{V}_h}=\left\{\boldsymbol{w}\in\uline{L}^2(\Omega);\  w_i|_K\in span\{1,x_i\}  \  \forall K\in\mathcal{T}_h\right\}.
\end{align*}
The local nodal degrees of freedom for the stress tensor   are shown in \cref{Hu-Man-Zhang}.

\end{itemize}
For the numerical quadrature on each element $K$, we divide  $K$ into two triangles and use the seven-points Gauss quadrature formula on each triangle.

\begin{figure}[h]
	\centering
	\includegraphics[width=11.5cm,height=4.5cm]{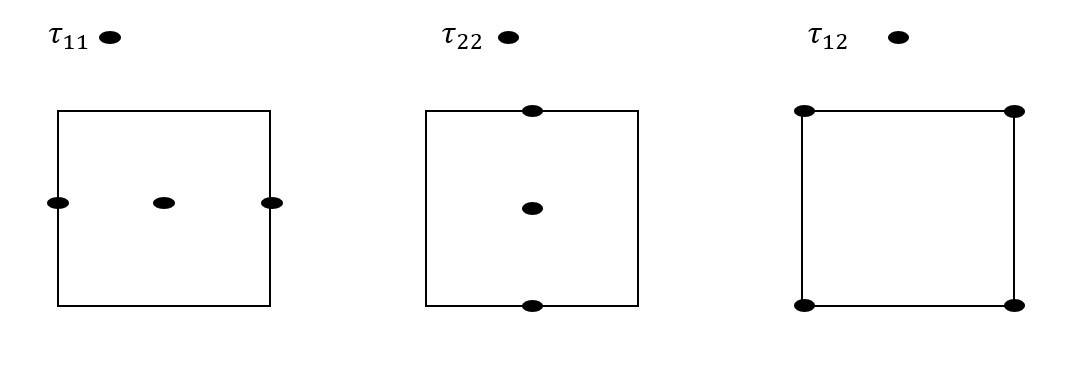}
	\caption{Nodal degrees of freedom for Hu-Man-Zhang's element}
	\label{Hu-Man-Zhang}
\end{figure}

 In the model  problem \cref{1.1}, we take $\Omega=[0,1]\times[0,1]$, $T=1$, and   assume that the elastic medium is isotropic with  $\rho=1, \mu=1, \lambda=1$.
 We use  $N\times N$ square meshes and $M$ uniform grids for   the spatial region $\Omega$ and the time region $[0,T]$.  To test the accuracy, we compute the following errors for the  stress and velocity approximations:
\begin{align*}
	&E^{a}_{\sigma}=\mathop{\mathrm{max}}_{1\leq n\leq M}||\sigma(t_{n})-\sigma_h(t_{n})||_a, \\
	&E^{c}_{\boldsymbol{v}}=\mathop{\mathrm{max}}_{1\leq n\leq M}||\boldsymbol{v}(t_{n})-\boldsymbol{v}_h(t_{n})||_c.
\end{align*}
According to \cref{theorem4.2,rem4.1}, the theoretical accuracy of the full discretization is 
\begin{align*}
	E^{a}_{\sigma}+E^{c}_{\boldsymbol{v}}\lesssim h+ (\Delta t)^2\approx N^{-1}+M^{-2}.
\end{align*}
 We consider the following  three examples. 

\begin{example}\label{example1}
	The exact   displacement field $\boldsymbol{u}(x,y,t)$ and symmetric stress tensor $\sigma(x,y,t)=(\sigma_{ij})_{2\times2}$ are respectively given by 
	\begin{equation*}
		\boldsymbol{u}=\left(
		\begin{array}{l}
			-e^{-t}(x^4-2x^3+x^2)(4y^3-6y^2+2y)\\
		-e^{-t}(y^4-2y^3+y^2)(4x^3-6x^2+2x)
		\end{array}
		\right),
	\end{equation*}
	 \begin{equation*}
\left(
\begin{array}{c}
	\sigma_{11} \\
	\sigma_{12} \\
	\sigma_{22} \\
\end{array}
\right)=\left(
	\begin{array}{l}
		16te^{-t}(2x^3-3x^2+x)(2y^3-3y^2+y) \\ 2te^{-t}[(x^4-2x^3+x^2)(6y^2-6y+1)+(y^4-2y^3+y^2)(6x^2-6x+1)] \\
        16te^{-t}(2x^3-3x^2+x)(2y^3-3y^2+y) \\
	\end{array}
	\right).
\end{equation*}
	Notice that the velocity field $\boldsymbol{v} =\boldsymbol{u}_t $. Numerical results  of $E^{a}_{\sigma}$ and $E^{a}_{\boldsymbol{v}}$ are shown in \cref{table1,table7}.

\end{example}

\begin{example}\label{example2}
	The exact   displacement field $\boldsymbol{u} $ and symmetric stress tensor $\sigma$ are respectively given by 	\begin{equation*}
		\boldsymbol{u}=\left(
		\begin{array}{l}
			-e^{-t}\sin(\pi x)\sin(\pi y)\\
			-e^{-t}\sin(\pi x)\sin(\pi y)
		\end{array}
		\right),
	\end{equation*}
	\begin{equation*}
	\left(
	\begin{array}{c}
		\sigma_{11} \\
		\sigma_{12} \\
		\sigma_{22} \\
	\end{array}
	\right)=\left(
	\begin{array}{l}
		\pi te^{-t}(3\cos(\pi x)\sin(\pi y)+\sin(\pi x)\cos(\pi y)) \\
		\pi te^{-t}(\sin(\pi x)\cos(\pi y)+\cos(\pi x)\sin(\pi y))\\
		\pi te^{-t}(3\sin(\pi x)\cos(\pi y)+\cos(\pi x)\sin(\pi y)) 
	\end{array}
	\right).
\end{equation*}
Numerical results are shown in \cref{table2,table8}.
	
\end{example} 

\begin{example}\label{example3}
The exact   displacement field $\boldsymbol{u} $ and symmetric stress tensor $\sigma$ are respectively given by 
	\begin{equation}
		\boldsymbol{u} =\left(
		\begin{array}{l}
			e^t\sin(\pi x)(y^{3/2}-y^{5/2})\\
			e^t\sin(\pi y)(x^{3/2}-x^{5/2})
		\end{array}
		\right),
	\end{equation}
	 \begin{equation}
	\left(
	\begin{array}{c}
		\sigma_{11} \\
		\sigma_{12} \\
		\sigma_{22} \\
	\end{array}
	\right)=\left(
	\begin{array}{l}
		\pi e^{t}(\frac{3}{2}\cos(\pi x)(y^{\frac{3}{2}}-y^{\frac{5}{2}})+\frac{1}{2}\cos(\pi y)(x^{\frac{3}{2}}-x^{\frac{5}{2}})) \\
		\frac{1}{2}e^t(\sin(\pi x)(\frac{3}{2}y^{\frac{1}{2}}-\frac{5}{2}y^{\frac{3}{2}})+\sin(\pi y)(\frac{3}{2}x^{\frac{1}{2}}-\frac{5}{2}x^{\frac{3}{2}})) \\
		\pi e^t(\frac{3}{2}\cos(\pi y)(x^{\frac{3}{2}}-x^{\frac{5}{2}})+\frac{1}{2}\cos(\pi x)(y^{\frac{3}{2}}-y^{\frac{5}{2}}))
	\end{array}
	\right).
\end{equation}
Numerical results are shown in \cref{table3,table9}.
\end{example}

\cref{table1,table2,table3} give some numerical results with a fixed  time step $\Delta t=0.005$ to verify the theoretical first order spatial-accuracy of the schemes.   \cref{table7,table8,table9} give   numerical results with synchronous refinement of  temporal and spatial meshes, $h=4(\Delta t)^2$ or equivalently $N=M^2/4$,   to verify the theoretical second order temporal-accuracy.  From all the numerical results we have the following observations:

\begin{itemize}

 \item As shown in \cref{table1,table2,table3},  the HMZ element  is of  first order spatial accuracy,   and 
 the Nedelec's $Q_{1}^{\mathrm{\textbf{div}}}-Q_0$ element  gives   better convergence rates than the first order for both  $E^{a}_{\sigma}$ and $E^{c}_{\boldsymbol{v}}$.   

 \item   As shown in  \cref{table7,table8,table9}, the HMZ element  is of  second order temporal-accuracy, and the
    $Q_{1}^{\mathrm{\textbf{div}}}-Q_0$ element  yields   higher than 2nd order  convergence rates. 
    
    \item For  the  $Q_{1}^{\mathrm{\textbf{div}}}-Q_0$ element,  the better convergence behaviours  than the theoretical prediction may be due to some superconvergence of the element on square meshes.
  
\end{itemize}

\begin{table}[H]
	\caption{Convergence history: \cref{example1} with $\Delta t=0.005$.}
	\label{table1}
\setlength{\tabcolsep}{4mm}
		\begin{tabular}{|c|c|c|c|c|c|}
		\hline  
		\multirow{2}{*}{}&\multirow{2}{*}{$N\times N$}&
		\multicolumn{2}{c|}{$E^{a}_{\sigma}$}&\multicolumn{2}{c|}{$E^{c}_{\boldsymbol{v}}$}\cr\cline{3-6}  
		& &error&order&error&order\cr  
		\hline
		\multirow{5}{*}{$Q_{1}^{\mathrm{\textbf{div}}}-Q_0$}&4$\times$4&0.0207&-&0.0066&- \\
		\cline{2-6}
		& 8$\times$8&0.0111&0.89&0.0033&0.98 \\
		\cline{2-6}
		& 16$\times$16&0.0053&1.08&0.0016&1.08\\
		\cline{2-6}
		&32$\times$32&0.0019&1.49&0.0007&1.11\\
		\cline{2-6}
		&64$\times$64&0.0004&1.97&0.0003&1.22\\
		\hline
        \multirow{5}{*}{HMZ}&4$\times$4&0.0097&-&0.0032&- \\
        \cline{2-6}
        & 8$\times$8&0.0054&0.83&0.0018&0.86 \\
        \cline{2-6}
        & 16$\times$16&0.0028&0.96&0.0008&0.97\\
        \cline{2-6}
        &32$\times$32&0.0014&0.99&0.0004&0.99\\
        \cline{2-6}
        &64$\times$64&0.0007&1.00&0.0002&1.00\\
        \hline
	\end{tabular}
\end{table}

\begin{table}[H]
	\caption{\label{table2}Convergence history: \cref{example2} with $\Delta t=0.005$.}
	\setlength{\tabcolsep}{4mm}
	\begin{tabular}{|c|c|c|c|c|c|}
		\hline  
		\multirow{2}{*}{}&\multirow{2}{*}{$N\times N$}&
		\multicolumn{2}{c|}{$E^{a}_{\sigma}$}&\multicolumn{2}{c|}{$E^{c}_{\boldsymbol{v}}$}\cr\cline{3-6}  
		& &error&order&error&order\cr  
		\hline
		\multirow{5}{*}{$Q_{1}^{\mathrm{\textbf{div}}}-Q_0$}&4$\times$4&0.9423&-&0.4120&- \\
		\cline{2-6}
		& 8$\times$8&0.5323&0.82&0.1925&1.10 \\
		\cline{2-6}
		& 16$\times$16&0.2245&1.25&0.0862&1.16\\
		\cline{2-6}
		&32$\times$32&0.0663&1.76&0.0355&1.28\\
		\cline{2-6}
		&64$\times$64& 0.0157&2.08&0.0156&1.18\\
		\hline
		\multirow{5}{*}{HMZ}&4$\times$4&0.3524&-&0.1587&- \\
		\cline{2-6}
		& 8$\times$8&0.1784&0.98&0.0797&0.99 \\
		\cline{2-6}
		& 16$\times$16&0.0896&0.99&0.0399&1.00\\
		\cline{2-6}
		&32$\times$32&0.0448&1.00&0.0199&1.00\\
		\cline{2-6}
		&64$\times$64&0.0224&1.00&0.0100&1.00\\
		\hline 		
	\end{tabular}
\end{table} 

\begin{table}[H]
	\caption{\label{table3}Convergence history: \cref{example3} with $\Delta t=0.005$.}
	\setlength{\tabcolsep}{4mm}
	\begin{tabular}{|c|c|c|c|c|c|}
		\hline  
		\multirow{2}{*}{}&\multirow{2}{*}{$N\times N$}&
		\multicolumn{2}{c|}{$E^{a}_{\sigma}$}&\multicolumn{2}{c|}{$E^{c}_{\boldsymbol{v}}$}\cr\cline{3-6}  
		& &error&order&error&order\cr  
		\hline
		\multirow{5}{*}{$Q_{1}^{\mathrm{\textbf{div}}}-Q_0$}&4$\times$4&0.6230&-&0.2136&- \\
		\cline{2-6}
		& 8$\times$8&0.3480&0.84&0.1000&1.10 \\
		\cline{2-6}
		& 16$\times$16&0.1522&1.19&0.0422&1.25\\
		\cline{2-6}
		&32$\times$32&0.0460&1.73&0.0179&1.24\\
		\cline{2-6}
		&64$\times$64&0.0111&2.05&0.0081&1.14\\
		\hline
		\multirow{5}{*}{HMZ}&4$\times$4&0.2531&-&0.0833&- \\
		\cline{2-6}
		& 8$\times$8&0.1307&0.95&0.0425&0.97 \\
		\cline{2-6}
		& 16$\times$16&0.0661&0.98&0.0214&0.99\\
		\cline{2-6}
		&32$\times$32&0.0332&0.99&0.0107&1.00\\
		\cline{2-6}
		&64$\times$64&0.0167&1.00&0.0053&1.00\\
		\hline 		
	\end{tabular}
\end{table}

\begin{table}[H]
	\caption{\label{table7}Convergence history: \cref{example1} with $N=M^2/4.$}
	\setlength{\tabcolsep}{4mm}
	\begin{tabular}{|c|c|c|c|c|c|}
		\hline  
		\multirow{2}{*}{}&\multirow{2}{*}{$M$}&
		\multicolumn{2}{c|}{$E^{a}_{\sigma}$}&\multicolumn{2}{c|}{$E^{c}_{\boldsymbol{v}}$}\cr\cline{3-6}  
		& &error&order&error&order\cr  
		\hline
		\multirow{4}{*}{$Q_{1}^{\mathrm{\textbf{div}}}-Q_0$}&4&0.0096&-&0.0054&- \\
		\cline{2-6}
		&8 &0.0014&2.76&0.0013&2.02 \\
		\cline{2-6}
		&12 &0.0004&2.91&0.0005&2.27\\
		\cline{2-6}
		&16 &0.0001&2.97&0.0002&2.22\\
		\hline
		\multirow{4}{*}{HMZ}&4&0.0097&-&0.0025&- \\
		\cline{2-6}
		&8 &0.0028&1.79&0.0007&1.66 \\
		\cline{2-6}
		&12 &0.0013&1.98&0.0003&1.88\\
		\cline{2-6}
		&16 &0.0007&2.00&0.0002&1.93\\
		\hline
	\end{tabular}
\end{table}

\begin{table}[H]
	\caption{\label{table8}Convergence history:  \cref{example2} with $N=M^2/4.$}
	\setlength{\tabcolsep}{4mm}
	\begin{tabular}{|c|c|c|c|c|c|}
		\hline  
		\multirow{2}{*}{}&\multirow{2}{*}{$M$}&
		\multicolumn{2}{c|}{$E^{a}_{\sigma}$}&\multicolumn{2}{c|}{$E^{c}_{\boldsymbol{v}}$}\cr\cline{3-6}  
		& &error&order&error&order\cr  
		\hline
		\multirow{4}{*}{$Q_{1}^{\mathrm{\textbf{div}}}-Q_0$}&4&0.3531&-&0.3231&- \\
		\cline{2-6}
		&8 &0.0404&3.13&0.0667&2.28 \\
		\cline{2-6}
		&12 &0.0118&3.03&0.0268&2.24\\
		\cline{2-6}
		&16 &0.0049&3.05&0.0145&2.14\\
		\hline
		\multirow{4}{*}{HMZ}&4&0.3536&-&0.1253&- \\
		\cline{2-6}
		&8 &0.0896&1.98&0.0354&1.82 \\
		\cline{2-6}
		&12 &0.0399&2.00&0.0164&1.90\\
		\cline{2-6}
		&16 &0.0224&2.00&0.0094&1.93\\
		\hline
	\end{tabular}
\end{table} 

\begin{table}[H]
	\caption{\label{table9}Convergence history: \cref{example3} with $N=M^2/4.$}
	\setlength{\tabcolsep}{4mm}
	\begin{tabular}{|c|c|c|c|c|c|}
		\hline  
		\multirow{2}{*}{}&\multirow{2}{*}{$M$}&
		\multicolumn{2}{c|}{$E^{a}_{\sigma}$}&\multicolumn{2}{c|}{$E^{c}_{\boldsymbol{v}}$}\cr\cline{3-6}  
		& &error&order&error&order\cr  
		\hline
		\multirow{4}{*}{$Q_{1}^{\mathrm{\textbf{div}}}-Q_0$}&4&0.2601&-&0.2126&- \\
		\cline{2-6}
		&8 &0.0362&2.85&0.0396&2.42 \\
		\cline{2-6}
		&12 &0.0117&2.78&0.0153&2.34\\
		\cline{2-6}
		&16 &0.0055&2.60&0.0081&2.21\\
		\hline
		\multirow{4}{*}{HMZ}&4&0.2528&-&0.0837&- \\
		\cline{2-6}
		&8 &0.0661&1.94&0.0215&1.96 \\
		\cline{2-6}
		&12 &0.0295&1.99&0.0096&2.00\\
		\cline{2-6}
		&16 &0.0166&1.99&0.0054&2.00\\
		\hline
	\end{tabular}
\end{table}

\bibliographystyle{plain}
\bibliography{viscoelastic}

\end{document}